\documentclass[a4paper,11pt]{amsart}
\usepackage{amsfonts}
\usepackage{amssymb}
\usepackage[utf8]{inputenc}
\usepackage{amsmath}
\usepackage{pdflscape}
\usepackage{easyReview}
\usepackage{graphicx}
\usepackage[colorlinks=true, linkcolor=red, citecolor=blue]{hyperref}
\usepackage[]{epsfig}
\usepackage[]{pstricks}
\usepackage{tikz}

\newtheorem{theorem}{Theorem}[section]
\newtheorem{proposition}[theorem]{Proposition}

\newtheorem{remark}{Remark}
\theoremstyle{remark}

\usepackage{mathrsfs}

\numberwithin{equation}{section}
\makeatletter
\@namedef{subjclassname@2010}{\textup{2020} Mathematics Subject Classification}
\makeatother

\begin{document}
	
	\pagenumbering{arabic}	
	\title[Feedback boundary stabilization for the Hirota-Satsuma system with time-delay]{Feedback boundary stabilization for the Hirota-Satsuma system with time-delay}
	\author[Gonzalez Martinez]{Victor Hugo Gonzalez Martinez}
	\author[Muñoz]{Juan Ricardo Muñoz*}
	\address{Departamento de Matem\'atica,  Universidade Federal de Pernambuco (UFPE), 50740-545, Recife (PE), Brazil.}
	\email{juan.ricardo@ufpe.br}
	
	\email{victor.martinez@ufpe.br}
	\thanks{*Corresponding author: juan.ricardo@ufpe.br}
	\subjclass[2010]{35Q53, 93D15, 93D30, 93C20}
	\keywords{Hirota-Satsuma system $\cdot$ Stabilization $\cdot$ Decay rate $\cdot$ Lyapunov approach $\cdot$ Observability}
	
	\begin{abstract}
		This work investigates the boundary stabilization problem of the Hirota-Satsuma system. In the problem under consideration, a boundary feedback law consisting of a linear combination of a damping mechanism and a time-delay term is designed. The study shows that, with time delay feedback and a smallness restriction on the size of the initial data the energy of the Hirota-Satsuma system decays exponentially by employing two approaches: the Lyapunov method and  an observability inequality combined with a contradiction argument.
	\end{abstract}
	\maketitle
	\section{Introduction}
	The Hirota-Satsuma system serves as a model for understanding certain types of nonlinear wave interactions and phenomena that arise in the propagation of nonlinear waves in shallow water or the behavior of waves in stratified fluids, in various physical systems and their properties provides insights into the behavior of nonlinear waves across different physical systems. 
	
	In 1981, Hirota and Satsuma~\cite{HS1} introduced a system that includes two real functions, depending on time and space, denoted by $u = u(t,x)$ and $v = v(t,x)$, modeling the interactions of two long waves with different dispersion relations.  The system is given by the equations:
	\begin{equation*}
		\begin{cases}
			u_t - a (u_{xxx} + 6uu_x) - 2b vu_x = 0, & x \in\mathbb{R},\  t\geq 0, \\
			v_t + v_{xxx} + 3uvx = 0, & x \in\mathbb{R},\  t\geq 0.
		\end{cases}
	\end{equation*}
	
	The asymptotic behavior of dispersive systems, described by partial differential equations (PDEs), has been a significant research focus recently. The main goal has been to develop control mechanisms, such as feedback and boundary controls, to stabilize these systems by ensuring energy decay or mitigating disturbances. Significant progress has been made in stabilizing systems on bounded domains, like KdV, Kawahara, and Boussinesq type systems, often achieving exponential stabilization through damping mechanisms. Additionally, the study of PDEs with time delays has gained attention due to their relevance in various fields, such as biology and engineering. Time delays, caused by factors like measurement lag or computation time, can both destabilize a system and improve its performance, depending on their implementation (See~\cite{Capistrano2012, Gallego2018, Capistrano2023, Capistrano2019, Datko88, Datko86, Nicaise2006, Parada2023, Pazoto2008, Rosier, Valein2022} and therein).
	
	Physically, the Hirota-Satsuma system specifically models two waves with different speeds, while the KdV-KdV system typically involves two waves with symmetric interaction. It is noteworthy that the Boussinesq KdV-KdV type system with boundary time-dependent delay was studied by the authors in~\cite{Munoz2024} obtaining the exponential decay. However, due to the lack of regularity of this system, only the linearized version admits this property. Inspired by this issue, our main objective is to obtain the global wellposedness and then, describe the asymptotic behavior with a boundary time-delay feedback for a KdV type system that includes the coupled nonlinear terms.
	
	Now, let's highlight the main novelties presented here:
	\begin{itemize}
		\item The Kato Smoothing effect provides consistent regularity, which is not time-dependent. As a result, regularity can be maintained at all positive times, allowing us to establish global wellposedness.
		
		\item The first method, inspired by the abstract context for delay-abstract differential equations~\cite{Nolasco}, allows us to prove the exponential stability and determine an explicit decay rate for the Hirota-Satsuma system. It's important to note that this approach requires less information on the traces.
		
		\item The second method is based on the Observability inequality and a contradiction argument, resulting in a uniform but generic exponential decay rate. We emphasize the similarity with the wave equation~\cite{Bardos} in the sense that we need to 'observe the boundary from all sides' to obtain the result, which implies more information about the traces on the system.
		
		\item  Exponential stability is achieved in both frameworks for the nonlinear system and the results are not restricted by the size of the spatial interval. Additionally, the constructive technique presented here can be adapted for a time-varying delayed system as in \cite{Munoz2024} and the general framework can be extended to other feedback mechanism.
	\end{itemize}

	\subsection{Problem Setting and main results}
	Let us describe the problem which we are interested in studying. Consider the bounded domain $(0,L)$ with $L>0$ and $t>0$. Then, the Hirota-Satsuma system is given by
	\begin{equation}\label{eq:HS}
		\begin{cases}
			u_t-\frac{1}{2} u_{x x x}-3 u u_x - 3 v v_x = 0 & x \in (0,L),\  t>0, \\
			v_t+v_{x x x}+3 u v_x = 0 & x \in (0,L),\  t>0, \\
			u(t,0) = u(t,L) = v(t,0) = v(t,L) = u_x(t,0) = 0, & t>0, \\
			v_x(t,L) = \alpha u_x(t,L) + \beta u_x(t-h,L), & t>0,\\
			u(0,x) = u_0(x), v(0,x) = v_0(x), &  x \in (0,L) \\
			u_x(t-h, L) = z_0(t-h,L), &  t \in (0,1).		
		\end{cases}
	\end{equation}
	that involves the parameters $\alpha$ and $\beta$ that will be related to the feedback gains given from the damping and anti-damping mechanism as the constant time-delay that will be denoted by $h$. Furthermore, the interaction between the feedback gains, where $\alpha$ and $\beta >0$ must satisfy the following constraint
	\begin{equation}\label{eq:cond}
		0 < 2\alpha^2 +\frac{3}{2}\beta< \frac{1}{2}.
	\end{equation}
	Then, we can define the total energy associated with the Hirota-Satsuma system~\eqref{eq:HS} as
	\begin{equation}\label{eq:En}
		E(t) = \frac{1}{2} \int_0^L u^2(t,x) + v^2(t,x)\,dx + \frac{\beta}{2}h \int_0^1 u_x^2(t-h\rho,L)\,d\rho. 
	\end{equation}
	Formally, some integrations by parts allow us to  deduce that
	\begin{equation}\label{eq:D.En.0}
		\begin{aligned}
			\frac{\mathrm{d}}{\mathrm{d} t} E(t)  
			& = \left(\frac{\beta}{2} -\frac{1}{4}\right)   u_x^2(t,L)  - \frac{\beta}{2} u_x^2(t-h,L) + \frac{1}{2}v_x^2(t,L) - \frac{1}{2}v_x^2(t,0) \\
			& =  \left(\frac{\beta}{2} -\frac{1}{4}\right)   u_x^2(t,L)  - \frac{\beta}{2} u_x^2(t-h,L) + \frac{1}{2}\left(\alpha u_x(t,L) + \beta u_x(t-h,L)\right)^2 - \frac{1}{2}v_x^2(t,0) .
		\end{aligned}
	\end{equation}
	By using the boundary conditions and writing in a matrix set up, yields that
	\begin{equation}\label{eq:D.En}
		\begin{aligned}
			\frac{\mathrm{d}}{\mathrm{d} t} E(t) 
			& = \frac{1}{2} \begin{pmatrix} u_x(t,L) \\ u_x(t-h, L)\end{pmatrix}^T \Phi_{\alpha,\beta}   \begin{pmatrix}u_x(t,L) \\ u_x(t-h, L)\end{pmatrix} 
			{ - \frac{1}{2}v_x^2(t,0) } \\
			& \leq \frac{1}{2} \begin{pmatrix} u_x(t,L) \\ u_x(t-h, L)\end{pmatrix}^T \Phi_{\alpha,\beta}   \begin{pmatrix}u_x(t,L) \\ u_x(t-h, L)\end{pmatrix} 
		\end{aligned}
	\end{equation}
	where
	\begin{equation*}\label{eq:Phi}
		\Phi_{\alpha,\beta} = \begin{pmatrix}
			\alpha^2 - \frac{1}{2} + \beta & \alpha\beta \\ \alpha\beta & \beta^2 -  \beta
		\end{pmatrix}
	\end{equation*}
	is a negative definite matrix. Indeed, the first entry satisfies
	\begin{equation*}
		\frac{1}{2} > \alpha^2 + \frac{3}{2}\beta > \alpha^2 + \beta \implies \alpha^2 + \beta - \frac{1}{2} < 0.
	\end{equation*}
	By using~\eqref{eq:cond} we get $\frac{1}{2} + \beta^2 > \alpha^2 + \frac{3}{2}\beta$, then the determinant is such that
	\begin{equation*}
		\begin{aligned}
			\frac{1}{\beta}\det \Phi 
			& = \left(\alpha^2 - \frac{1}{2} +\beta\right)\left(\beta -1\right) -\alpha^2\beta \\
			& = -\alpha^2 - \frac{3}{2} \beta + \beta^2 + \frac{1}{2} >0.
		\end{aligned}
	\end{equation*}
	
	Therefore, from~\eqref{eq:D.En}, we obtain that the total energy $E(t)$ associated with the Hirota-Satsuma system is a non-increasing function. This leads to the natural question:
	
	\begin{center}
		\textit{
			Does $E(t) \to 0$ as $t\to\infty$? If this happens, can we determine an explicit decay rate?}
	\end{center}
	
	Before presenting our positive answer to this question and the main result of this work let us define the functional spaces that will be used throughout the analysis, $X_0 := [L^2(0,L)]^2$ and $H:= X_0 \times L^2(0,1)$ and consider
	\begin{equation*}
		\mathcal{B} := C([0,T], X_0) \cap L^2(0,T, [H^1(0,L)]^2)
	\end{equation*}
	with the associated norm
	$\| (u,v)\|_{\mathcal{B}} = \sup_{t\in[0,T]} \| (u(t), v(t)\|_{X_0} + \| (u_x, v_x)\|_{L^2(0,T , X_0)}.$
	
	Then, by a constructive approach based on a energy perturbation argument, we ensure that the total energy associated with the Hirota-Satsuma system~\eqref{eq:HS} decays exponentially, that is,
	\begin{theorem}\label{th:ExpDec} 
		Let $L>0$ and $\alpha,\beta$ such that~\eqref{eq:cond} yields. Then, there exists $0 < r < 3/16L^{\frac{3}{2}}$	such that for every initial data $(u_0,v_0,z_0) \in H$ with $\| (u_0,v_0,z_0)\|_{H}\leq r,$ the energy $E(t)$ defined in~\eqref{eq:En} of the Hirota-Satsuma system~\eqref{eq:HS} decays exponentially. More precisely, given  $\mu_1, \mu_2$ positive constants small enough, then there exists $\kappa = 1+\max\lbrace \mu_1 L,\mu_2 \rbrace$ and
		\begin{equation*}\label{eq:lambda}
			\lambda \leq \min\left\lbrace \frac{\pi^2 \mu_1(3-16L^{\frac{3}{2}}r)}{2L^2(1+L\mu_1)},\frac{\mu_2}{h(1+\mu_2)} \right\rbrace
		\end{equation*}
		such that  $E(t) \leq \kappa E(0) e^{-\lambda t}$, for all $t \geq 0$.
	\end{theorem}
	
	Moreover, by using a general framework	known as Hilbert Uniqueness Method, that relies on an observability inequality and a contradiction argument, we can obtain a generic exponential stability,
	\begin{theorem}\label{th:ES.2}
		There exist two constant $\kappa_0, \lambda_0 > 0$ such that for any $(u_0,v_0,z_0) \in H$ satisfying $\| (u_0,v_0,z_0)\|_H \leq r$, the energy of the system~\eqref{eq:En} satisfies
		\begin{equation*}
			E(t) \leq \kappa_0 E(0) e^{-\lambda_0 t},\quad \forall t\geq 0.
		\end{equation*}
	\end{theorem}
	
	We end this section by providing an outline of this paper: Section~\ref{sec:2} is devoted to the proofs of the well-posedness, first dealing with the linear system and then by the Kato smoothing addressing the global well-posedness. Next, in Section~\ref{sec:3} we use Lyapunov's approach and the observability inequality to obtain the exponential stability of the solutions of the Hirota-Satsuma system issued from small initial data.

	\section{Well-posedness}\label{sec:2}
	In this section, we address the well-posedness of the Hirota-Satsuma system, that is, we obtain the existence of solutions, that in conjunction with some \emph{a priori} estimates and the Kato smoothing effect, allow us to deal with the nonlinear systems and obtain a prove the global well-posedness of the solutions.
	\subsection{Linear problem}
	First, we consider the linearization of~\eqref{eq:HS} around the origin
	\begin{equation}\label{eq:HS.lin}
		\begin{cases}
			u_t - \frac{1}{2} u_{x x x} = 0 & x \in (0,L),\  t>0 \\
			v_t+v_{x x x} = 0 & x \in (0,L),\  t>0 \\
			u(t,0) = u(t,L) = v(t,0) = v(t,L) = u_x(t,0) = 0, & t>0, \\
			v_x(t,L) = \alpha u_x(t,L) + \beta u_x(t-h,L), & t>0,\\
			u(0,x) = u_0(x), v(0,x) = v_0(x) \in L^2(0,L) \\
			u_x(t-h, L) = z_0(t-h,L) \in L^2(0,1).
		\end{cases}
	\end{equation}
	
	Now, following the idea introduced in~\cite{Nicaise2006}, let us introduce the change of variables $z(t,\rho) = u_x(t-h\rho,L)$ with $\rho\in(0,1)$ that satisfies the transport equation
	\begin{equation}\label{eq:tr}
		\begin{cases}
			h z_t(t,\rho) +z_\rho(t,\rho) = 0, & \rho \in(0,1),\ t>0 \\
			z(t,0) = u_x(t,L), z(0,\rho) = z_0(-h\rho), & \rho \in(0,1),\ t>0
		\end{cases}
	\end{equation}
	and consider $H$ equipped with the inner product\footnote{This new inner product is clearly equivalent to the usual inner product on $H$}
	\begin{equation*}
		\left\langle (u,v, z) , (\overline{u}, \overline{v}, \overline{z}) \right\rangle = \left\langle (u,v) , (\overline{u}, \overline{v}) \right\rangle_{X_0} + \beta h\left\langle z,\overline{z} \right\rangle_{L^2(0,1)}
	\end{equation*}
	for any $(u,v,z) , (\overline{u}, \overline{v}, \overline{z}) \in H$. Now, pick up $U = (u,v,z)$ and recast~\eqref{eq:HS.lin}-\eqref{eq:tr} as a Cauchy abstract problem
	\begin{equation*}\label{eq:Cauchy.abs}
		\dfrac{\mathrm{d}}{\mathrm{d}t} U = A U, \quad 		U(0) = U_0,  t>0,
	\end{equation*}
	where $A \colon D(A) \subset H \to H$ is the operator given  by
	\begin{equation*}\label{eq:A}
		A (u,v,z ) : = \left( \frac{1}{2}u_{xxx} ,- v_{xxx}, -\frac{1}{h} z_\rho \right)
	\end{equation*}
	with densely defined domain
	\begin{equation*}
		D(A) : =
		\left\lbrace
		\left.\begin{aligned}
			(u,v) \in [H^3(0,L) \cap H_0^1(0,L)]^2, \\
			z \in H^1(0,1)
		\end{aligned}\right|
		\begin{aligned}
			u_x(0) = 0,\ z(0) = u_x(L) , \\
			v_x(L) = \alpha u_x(L) + \beta z(1)
		\end{aligned}
		\right\rbrace \subset H
	\end{equation*}
	
	\begin{proposition}
		Suppose that~\eqref{eq:cond} yields. Then $A$ generates a continuous semigroup of contractions $(S(t))_{t \geq 0 }$ in $H$.
	\end{proposition}
	\begin{proof}
		Clearly, $A$ is densely defined and closed, so we are done if we prove that $A$ and its adjoint $A^\ast$ are both dissipative in $H$. It is readily seen that $A^\ast\colon D(A^\ast) \subset H \to H$ is given by
		\begin{equation}\label{eq:Aadj}
			A^\ast (\eta,\omega,\theta ) : = \left( -\frac{1}{2}\eta_{xxx} , \omega_{xxx}, \frac{1}{h} \theta_\rho \right)
		\end{equation}
		with  domain
		\begin{equation}
			D(A^\ast) : =
			\left\lbrace
			\left.\begin{aligned}
				(\eta,\omega) \in [H^3(0,L) \cap H_0^1(0,L)]^2, \\
				\theta \in H^1(0,1)
			\end{aligned}\right |
			\begin{aligned}
				\omega_x(0) = 0,\ \theta(1) = \omega_x(L) , \\
				\eta_x(L) = 2\alpha \omega_x(L) + 2\beta \theta(0)
			\end{aligned}
			\right\rbrace \subset H.
		\end{equation}
		
			
			Let $(u,v,z) \in D(A)$, then performing some integrations by parts holds that
			\begin{equation}
				\left\langle A (u,v,z); (u,v,z) \right\rangle_{H}  \leq \frac{1}{2} \begin{pmatrix} u_x(t,L) \\ u_x(t-h, L)\end{pmatrix}^T \Phi  \begin{pmatrix}u_x(t,L) \\ u_x(t-h, L)\end{pmatrix}^T  \leq 0.
			\end{equation}
			Where $\Phi_{\alpha,\beta}$ is the negative definite matrix given by~\eqref{eq:Phi}.
			On the other hand, let $(\eta,\omega,\theta) \in D(A^\ast)$, then
			\begin{equation}
				\left\langle A^\ast (\eta,\omega,\theta); (\eta,\omega,\theta) \right\rangle_{H}  \leq \frac{1}{2}\begin{pmatrix} \omega_x(L) \\ \theta(0) \end{pmatrix}^T \tilde\Phi_{\alpha,\beta}\begin{pmatrix} \omega_x(L) \\ \theta(0) \end{pmatrix} 
			\end{equation}
			where
			\begin{equation}
				\tilde\Phi =\begin{pmatrix}
					2\alpha^2 + \beta -1 & \alpha\beta \\ \alpha\beta & 2\beta^2-\beta
				\end{pmatrix}.
			\end{equation}
			It is not difficult to verify that under the assumption~\eqref{eq:cond}, $\tilde\Phi_{\alpha,\beta}$ is negative definite and consequently
			\begin{equation}
				\left\langle A^\ast (\eta,\omega,\theta); (\eta,\omega,\theta) \right\rangle_{H}  \leq \frac{1}{2}\begin{pmatrix} \omega_x(L) \\ \theta(0) \end{pmatrix}^T \tilde\Phi\begin{pmatrix} \omega_x(L) \\ \theta(0) \end{pmatrix} \leq 0.
			\end{equation}
			
			Summarizing , $A$ and $A^\ast$ are dissipative, from  \cite[Corollary 4.4, page 15]{Pazy} the result yields.
		\end{proof}
		
		Notice that, the behavior of the energy $E(t)$ depends on the traces. Then we can establish the next proposition to state that the energy~\eqref{eq:En} is decreasing along the solutions of~\eqref{eq:HS.lin}.
		\begin{proposition}\label{pr:Diss}
			Suppose that $\alpha$ and $\beta$ are real constants such that~\eqref{eq:cond} holds. Then for any mild solution of~\eqref{eq:HS.lin} the energy $E(t)$ defined by~\eqref{eq:En} is non-increasing and there exists a constant $\mathcal{K} > 0$ such that
			\begin{equation}\label{eq:Diss}
				E'(t) 
				\leq - \mathcal{K}\left[u_x^2(t,L) + u_x^2(t-h,L) + v_x^2(t,0)\right]
			\end{equation}
			where $\mathcal{K}= \mathcal{K}(\alpha,\beta)$
		\end{proposition}

		The following proposition provides useful estimates for the mild solutions of~\eqref{eq:HS.lin}. The first ones are standard energy	estimates, while the last one reveals a Kato smoothing effect. 
		\begin{proposition}\label{pr:Kato}
			Let $\alpha$ and $\beta$ are real constant  such that~\eqref{eq:cond} holds. Then, the map
			\begin{equation*}
				(u_0,v_0,z_0)\in{H} \mapsto (u,v,z) \in \mathcal{B} \times C(0,T; L^2(0,1))
			\end{equation*}
			is well defined, continuous and fulfills
			\begin{equation}\label{eq:Kato1}
				\|(u,v)\|_{X_0}^2 +  \beta \|z\|_{L^2(0,1)}^2 \leq \|	(u_0,v_0)\|_{X_0}^2+  \beta\|z_0(-h\cdot)\|_{L^2(0,1)}^2.
			\end{equation}
			Furthermore, for every $(u_0,v_0, z_0)\in {H}$, we have that
			\begin{equation}\label{eq:Katotr0}
				\| u_x(\cdot,L) \|_{L^2(0,T)}^2+\| z(\cdot,1) \|_{L^2(0,T)}^2 \leq  \|(u_0,v_0)\|_{X_0}^2 +\| z_0(-h\cdot)\|_{L^2(0,1)}^2
			\end{equation}
			Moreover, the Kato smoothing effect is verified, that is,
			\begin{equation}\label{eq:Kato2}
				\int_0^T\int_0^L u_x^2+v_x^2 \,dx\,dt \leq C(L,T,\alpha,\beta) \left(\| (u_0,v_0) \|_{X_0}^2	+	\| z_0(-h\cdot) \|_{L^2(0,1)}^2\right).
			\end{equation}
		\end{proposition}
		\begin{proof}
			The proof of estimates~\eqref{eq:Kato1}-\eqref{eq:Katotr0} is analogous to Proposition 2.4 in \cite{Munoz2024}.
			
			Now, we use Morawetz multipliers technique. Multiplying~\eqref{eq:HS.lin}$_1$ by $(L-x)u$ and~\eqref{eq:HS.lin}$_2$ by $xv$ adding the results and integrating by parts follows that
			\begin{equation*}
				\begin{aligned}
					0 & = \frac{1}{2} \int_0^L (L-x)\left[u^2(T,x) - u_0^2(x)\right]\,dx + \frac{1}{2} \int_0^L x\left[v^2(T,x) - v_0^2(x)\right]\,dx \\
					& \quad  \frac{3}{4} \int_0^T\int_0^L u_x^2\,dx\,dt+ \frac{3}{2} \int_0^T\int_0^L  v_x^2\,dx \,dt - \frac{L}{2} \int_0^T \left[\alpha u_x(t,L) + \beta z(1) \right]^2\,dt
				\end{aligned}
			\end{equation*}
			This implies, by using~\eqref{eq:Kato1} and~\eqref{eq:Katotr0}
			\begin{equation*}
				\begin{aligned}
					\frac{3}{4} \int_0^T\int_0^L u_x^2\,dx\,dt &+ \frac{3}{2} \int_0^T\int_0^L  v_x^2\,dx \,dt = \frac{1}{2} \int_0^L (L-x) \left[u_0^2 - u^2(T,x) \right]\,dx   \\
					& \quad   + \frac{1}{2} \int_0^L x \left[v_0^2 - v^2(T,x) \right]\,dx  + \frac{L}{2} \int_0^T \left[\alpha u_x(t,L) + \beta z(1) \right]^2\,dt \\
					&\leq \frac{L}{2} \| (u_0,v_0)\|_{X_0}^2 - \frac{L}{2} \| (u(T,x),v(T,x))\|_{X_0}^2 \\
					&\quad + (\alpha^2+\beta^2)L \left[\int_0^T u_x^2(t,L)\,dt + \int_0^T z^2(1)\,dt\right] \\
					& \leq L \| (u_0,v_0)\|_{X_0}^2 + (\alpha^2 + \beta^2) L \left(\| u_x(\cdot,L) \|_{L^2(0,T)}^2 + \| z(\cdot,1) \|_{L^2(0,T)}^2\right) \\
					& \leq C(L,\alpha,\beta)\left( \| (u_0,v_0)\|_{X_0}^2  + \| z_0(-h\cdot) \|_{L^2(0,1)}^2\right) 
				\end{aligned}
			\end{equation*}
			Consequently~\eqref{eq:Kato2} is verified with $C(L,\alpha,\beta) = \frac{4}{3} L(1+\alpha^2+\beta^2)$.
		\end{proof}
		
		\begin{remark}
			The regularity of the Hirota-Satsuma system differs from that of the KdV-KdV system due to its asymmetric structure. In the KdV-KdV system, the symmetric coupling allows for the use of symmetric Morawetz multipliers to achieve regularizing effects like the Kato smoothing effect. However, the asymmetric interaction in the Hirota-Satsuma system, particularly due to its nonlinear and linear coupling terms, makes symmetric multipliers insufficient for obtaining the same regularity results, requiring a different multipliers techniques for handling the smoothing effect.
		\end{remark}
		
		\subsection{Nonlinear problem}
		
		Here, we aim to obtain the well-posedness for the Hirota-Satsuma system \eqref{eq:HS}, we decompose the procedure in two steps. We start by turning our attention to consider the linear system~\eqref{eq:HS.lin} with source terms $f_1,f_2 \in L^1(0,T, X_0)$,
		\begin{equation}\label{eq:HS.Sou}
			\begin{cases}
				u_t - \frac{1}{2} u_{x x x} = f_1 & x \in (0,L),\  t>0 \\
				v_t+v_{x x x} = f_2 & x \in (0,L),\  t>0 \\
				u(t,0) = u(t,L) = v(t,0) = v(t,L) = u_x(t,0) = 0, & t>0, \\
				v_x(t,L) = \alpha u_x(t,L) + \beta u_x(t-h,L), & t>0,\\
				u(0,x) = u_0(x), v(0,x) = v_0(x) \in L^2(0,L) \\
				u_x(t-h, L) = z_0(t-h,L) \in L^2(0,1).
			\end{cases}
		\end{equation}
		
		By the Kato smoothing, we can ensure that the system is well-posed. More precisely, we have the following result:
		\begin{theorem}\label{th:KatoSou}
			Assume that~\eqref{eq:cond} holds. Let $U_0 = (u_0, v_0, z_0) \in H$ and the source terms $f_1,f_2 \in L^1(0,T, X_0)$. Then, there exists a unique solution $U = (u,v,z) \in C([0,T], H)$ to~\eqref{eq:HS.Sou}.
			Moreover, for $T>0$, there exists $C>0$ such that the following estimates hold
			\begin{equation*}\label{eq:KatoSou}
				\begin{aligned}
					\|(u,v,z)\|_{C([0,T], H)} &\leq  C\left( \|	(u_0,v_0,z_0)\|_{H}+ \| (f_1,f_2) \|_{L^1(0,T , X_0)}\right) , \\
					\left\| \left(u_x(\cdot,L),  z(\cdot,1)\right) \right\|_{L^2(0,T)}^2 &\leq  C\left(\|(u_0,v_0,z_0)\|_{H}^2 + \| (f_1,f_2) \|_{L^1(0,T , X_0)}^2\right), \\
					\left\|(u,v)\right\|_{L^2(0,T, [H^1(0,L)]^2)} &\leq C\left( \|	(u_0,v_0,z_0)\|_{H}+ \| (f_1,f_2) \|_{L^1(0,T , X_0)}\right).
				\end{aligned}
			\end{equation*}
		\end{theorem}
		\begin{proof}
			We can proceed as Theorem 2.5 in \cite{Munoz2024}.
		\end{proof}	

		In the second step, we can address the well-posedness of the nonlinear system~\eqref{eq:HS} by associating the source terms $(f_1,f_2)$ with the nonlinear terms $(uu_x +vv_x, uv_x)$. Essentially, we need to prove that the map $\Gamma \colon \mathcal{B} \to \mathcal{B}$ has a unique fixed-point in some closed ball $B(0, R)\subset \mathcal{B}$. This map is defined by $\Gamma(\tilde u,\tilde v) = (u,v)$, and $(u,v)$ are the solution of the system \eqref{eq:HS}. First, in the next Proposition we guarantee that the nonlinear terms can be considered a source term of the linear equation~\eqref{eq:HS.Sou}.
		
		\begin{proposition}\label{pr:nl.Sou}
			Let $(u,v)\in L^2(0,T, [H^1(0,L)]^2)$, so $uv_x$, $uu_x\in L^1(0,T, X_0)$ and
			$
			(u,v) \in\mathcal{B} \mapsto (uu_x +vv_x, uv_x)\in L^1(0,T, X_0)
			$
			is continuous. In addition, the following estimate holds,
			\begin{equation}\label{eq:nl.Sou.1}
				\begin{split}
					\int_0^T \left\|(u_1u_{1,x} + v_1v_{1,x} - (u_2u_{2,x} + v_2v_{2,x}) , u_1v_{1,x} - u_2v_{2,x}) \right\|_{X_0}\,dt \\
					\leq K\left(\| (u_1,v_1) \|_{\mathcal{B}} +\| (u_2,v_2)\|_{\mathcal{B}}\right) \| (u_1-u_2, v_1-v_2 ) \|_{\mathcal{B}}\,
				\end{split}
			\end{equation}
			for a constant $K>0$.
		\end{proposition}
		\begin{proof}
			By the Sobolev embbeding  $H^1(0,L) \hookrightarrow L^\infty(0,L)$, follows that
			\begin{equation*}
				\| uv_x\|_{L^2(0,L)} \leq \| u \|_{L^\infty(0,L)} \|v_x\|_{L^2(0,L)} \leq K \| u\|_{H^1(0,L)} \| v \|_{H^1(0,L)}.
			\end{equation*}
			Consequently, there exists a constant $K>0$ such that 
			\begin{multline*}
				\|(u_1u_{1,x} + v_1v_{1,x} - (u_2u_{2,x} + v_2v_{2,x}) , u_1v_{1,x} - u_2v_{2,x}) \|_{X_0} \\
				\leq K\left( \| (u_1, v_1)\|_{[H^1(0,L)]^2} + \| (u_2,v_2)\|_{[H^1(0,L)]^2} \right) \|(u_1-u_2, v_1 - v_2 )\|_{[H^1(0,L)]^2}.
			\end{multline*}
			Then, by integrating on $[0,T]$ and using the Cauchy-Schwarz inequality,~\eqref{eq:nl.Sou.1} holds.
		\end{proof}
		
		Finally, we are in a position to present the existence of solutions to the Hirota-Satsuma System~\eqref{eq:HS}.
		
		\begin{theorem}\label{th:NLinSol}
			Let $L, T> 0$ and consider $\alpha$ and $\beta$  real constants such that~\eqref{eq:cond} is satisfied. For each initial data $(u_0, v_0; z_0) \in  H$ sufficiently small,  $\Gamma\colon \mathcal{B} \to\mathcal{B}$ defined by $\Gamma(\tilde u,\tilde v) = (u,v)$ is a contraction. Moreover, there exists a unique solution $(u,v) \in B(0, R) \subset \mathcal{B}$   of the Hirota-Satsuma system \eqref{eq:HS}.
		\end{theorem}
		\begin{proof}
			It follows from Theorem~\ref{th:KatoSou} that the map $\Gamma$ is well defined. Using Proposition~\ref{pr:nl.Sou} and the \emph{a priori estimates}~\eqref{eq:KatoSou} we obtain that
			\begin{equation*}
				\lVert \Gamma(\tilde u,\tilde v)\rVert_{\mathcal{B}}
				= \lVert (u, v) \rVert_{\mathcal{B}}
				\leq C\left( \lVert( u_0, v_0, z_0)\rVert_{H} + \lVert (\tilde u,\tilde v)\rVert_{\mathcal{B}}^2 \right),
			\end{equation*}
			and 
			\begin{equation*}
				\left\lVert \Gamma(\tilde u_1,\tilde v_1) - \Gamma(\tilde u_2,\tilde v_2) \right\rVert_{\mathcal{B}}
				\leq
				K \left( \lVert (\tilde u_1,\tilde v_1) \rVert_{\mathcal{B}} + \lVert (\tilde u_2,\tilde v_2)\rVert_{\mathcal{B}}\right) \lVert (\tilde u_1-\tilde u_2, \tilde v_1-\tilde v_2 ) \rVert_{\mathcal{B}}.
			\end{equation*}
			Now, we restrict $\Gamma$ to the closed ball $\lbrace(\tilde u,\tilde v)\in\mathcal{B}: \lVert (\tilde u,\tilde v)\rVert_{\mathcal{B}}\leq R\rbrace$, with $R > 0$ to be determined later. Then,
			$
			\lVert \Gamma(\tilde u,\tilde v)\rVert_{\mathcal{B}}
			\leq C\left( \lVert( u_0, v_0, z_0)\rVert_{H} +R^2 \right)$
			and
			$$
			\left\lVert \Gamma(\tilde u_1,\tilde v_1) - \Gamma(\tilde u_2,\tilde v_2) \right\rVert_{\mathcal{B}}
			\leq
			2RK \lVert (\tilde u_1-\tilde u_2, \tilde v_1-\tilde v_2 ) \rVert_{\mathcal{B}}.
			$$
			Next, we pick $R = 2C\lVert( u_0, v_0, z_0)\rVert_{H} $ such that $2K R < 1$, with $C < 2K$. This leads to claim that $$\lVert \Gamma(\tilde u,\tilde v)\rVert_{\mathcal{B}}\leq R$$ and  $$\left\lVert \Gamma(\tilde u_1,\tilde v_1) - \Gamma(\tilde u_2,\tilde v_2) \right\rVert_{\mathcal{B}} < C_1  \lVert (\tilde u_1-\tilde u_2, \tilde v_1-\tilde v_2 ) \rVert_{\mathcal{B}},$$ with $C_1<1$. Finally, the result yields as consequence of the Banach fixed point theorem.
		\end{proof}
		
		\begin{remark}
			In contrast to the results obtained for the KdV-KdV system (see Remark 2 in \cite{Munoz2024}), it is important to note that the solutions of the Hirota-Satsuma system~\eqref{eq:HS}obtained in Theorem~\ref{th:NLinSol} are global. This is due to Proposition~\ref{pr:Diss}, which essentially stems from the fact that the nonlinearities maintain the non-increasing nature of the energy $E(t)$ for the nonlinear system~\eqref{eq:HS}.
		\end{remark}
		
		\section{Boundary exponential stabilization}\label{sec:3}
		
		\subsection{A constructive approach}
		By constructing an appropriate perturbation to the Energy, we can systematically analyze the stability properties of the Hirota-Satsuma system. This subsection will delve into applying of Lyapunov's approach, achieving the desired boundary exponential stabilization for the Hirota-Satsuma system.
		
		\begin{proof}[ Proof of Theorem~\ref{th:ExpDec}]
			Let us introduce the Lyapunov functional $V(t)$ defined as
			\begin{equation*}\label{eq:V}
				V(t) = E(t) + \mu_1 V_1(t) + \mu_2 V_2(t)
			\end{equation*}
			where $\mu_1,\mu_2 \in \mathbb{R}^+$ will be chosen later, $E(t)$ is the total energy given by~\eqref{eq:En},
			\begin{equation*}\label{eq:V1}
				V_1(t) = \frac{1}{2} \int_0^L (L-x)u^2(t,x) + xv^2(t,x)\,dx \hbox{ and } V_2(t) = \frac{\beta h}{2} \int_0^1 (1-\rho) u_x^2(t-h\rho,L) \,d\rho.
			\end{equation*}		
			Notice that $E(t)$ and $V(t)$ are equivalent in the sense, $ E(t)\leq V(t) \leq\left(1+\max\lbrace \mu_1L, \mu_2\rbrace \right)E(t).$ In order to obtain the exponential decay, we are going to estimate $V'(t) + \lambda V(t)$. 
			
			Using equation \eqref{eq:HS} and performing integration by parts we obtain 
			\begin{equation*}
				\begin{aligned}
					V_1'(t)  
					& =  - \frac{3}{4} \int_0^L u_x^2\,dx - \frac{3}{2} \int_0^L v_x^2\,dx + \frac{L}{2} \begin{pmatrix} u_x(t,L) \\ u_x(t-h, L)\end{pmatrix}^T  \begin{pmatrix} \alpha^2 & \alpha\beta \\ \alpha\beta & \beta^2\end{pmatrix} \begin{pmatrix} u_x(t,L) \\ u_x(t-h, L)\end{pmatrix} \\
					&\quad +  \int_0^L u^3\,dx + 3\int_0^L (L-2x)uvv_x\,dx.
				\end{aligned}
			\end{equation*}
			
			On the other hand, observe that
			\begin{equation*}
				\begin{aligned}
					V_2'(t)	& =  - \frac{\beta}{2} \int_0^1 u_x^2(t-h\rho,L) \,d\rho + \frac{1}{2} \begin{pmatrix} u_x(t,L) \\ u_x(t-h, L)\end{pmatrix}^T  \begin{pmatrix} \beta & 0 \\ 0 & 0 \end{pmatrix} \begin{pmatrix} u_x(t,L) \\ u_x(t-h, L)\end{pmatrix}
				\end{aligned}
			\end{equation*}

			Gathering all these results, follows that
			\begin{equation*}
				\begin{aligned}
					V'(t) + \lambda V(t) & \leq \frac{1}{2} \left\langle \Psi_{\mu_1,\mu_2} (u_x(t,L) , u_x(t-h\rho,L), (u_x(t,L) , u_x(t-h\rho,L) \right\rangle  - \frac{3}{4} \mu_1 \int_0^L u_x^2 + v_x^2\,dx \\
					&\quad		+ \frac{\lambda}{2}(1+L\mu_1) \int_0^L u^2 + v^2\,dx  +  \int_0^L \mu_1 u^3\,dx + 3\int_0^L \mu_1 (L-2x)uvv_x\,dx \\
					& \quad +\frac{\beta}{2} (\lambda h + \lambda h\mu_2 - \mu_2 ) \int_0^1 u_x^2(t-h\rho,L)\,d\rho.
				\end{aligned}
			\end{equation*}
			Here,
			\begin{equation}\label{eq:Psi}
				\Psi_{\mu_1,\mu_2} = \Phi_{\alpha,\beta}  + L\mu_1\begin{pmatrix} \alpha^2 & \alpha\beta \\ \alpha\beta & \beta^2\end{pmatrix} + \mu_2\begin{pmatrix} \beta & 0 \\ 0 & 0 \end{pmatrix}.
			\end{equation}
			Due the continuity of the trace and the determinant we can choose $\mu_1,\mu_2$ small enough~(see Remark~\ref{re:mu}) such that $\Psi_{\mu_1,\mu_2}$ is definite negative and consequently
			\begin{equation}\label{eq:Psi.1}
				\left\langle \Psi_{\mu_1,\mu_2} (u_x(t,L) , u_x(t-h\rho,L), (u_x(t,L) , u_x(t-h\rho,L) \right\rangle \leq 0.
			\end{equation}
			Then, by~\eqref{eq:Psi.1} and employing Poincaré's inequality holds that
			\begin{equation*}
				\begin{aligned}
					V'(t) + \lambda V(t) & \leq  \left[ \frac{\lambda L^2}{2\pi^2}(1+L\mu_1)  - \frac{3}{4} \mu_1\right]\int_0^L u_x^2 + v_x^2\,dx  +  \int_0^L\mu_1 u^3\,dx + 3\int_0^L \mu_1(L-2x)uvv_x\,dx \\
					& \quad +\frac{\beta}{2} (\lambda h + \lambda h\mu_2 - \mu_2 ) \int_0^1 u_x^2(t-h\rho,L)\,d\rho.
				\end{aligned}
			\end{equation*}
			Let us deal with the nonlinear terms, by  the Sobolev embbeding $H_0^1(0,L) \hookrightarrow L^\infty(0,L)$ and the generalized Hölder's inequality we obtain
			\begin{equation}\label{eq:nlT.1}
				\int_0^L u^3\,dx  \leq  \|u(t,\cdot)\|_{L^\infty(0,L)}^2 \int_0^L  |u|\,dx \leq L^{\frac{3}{2}} r  \|u_x(t,\cdot)\|_{L^2(0,L)}^2
			\end{equation}
			and
			\begin{equation}\label{eq:nlT.2}
				\begin{aligned}
					3\int_0^L (L-2x) uvv_x\,dx 
					& \leq 3L \|u(t,\cdot)\|_{L^2(0,L)}\|v(t,\cdot)\|_{L^\infty(0,L)}\|v_x(t,\cdot)\|_{L^2(0,L)} \\
					& \leq   3L^{\frac{3}{2}} r \|v_x(t,\cdot)\|_{L^2(0,L)}^2
				\end{aligned}
			\end{equation}
			
			Therefore, taking the constants $\lambda$ and $r$ as in the statement of Theorem \ref{th:ExpDec}, we obtain
			\begin{equation*}
				\begin{aligned}
					V'(t) + \lambda V(t) & \leq \left[ \frac{\lambda L^2}{2\pi^2}(1+L\mu_1)  - \frac{3}{4} \mu_1 + 4L^{\frac{3}{2}}\mu_1r\right]\int_0^L u_x^2 + v_x^2\,dx \\
					& \quad +\frac{\beta}{2} (\lambda h + \lambda h\mu_2 - \mu_2 ) \int_0^1 u_x^2(t-h\rho,L)\,d\rho \leq 0.
				\end{aligned}
			\end{equation*}
			Consequently, by using Gronwall's inequality the result yields. 
		\end{proof}

		\begin{remark}\label{re:mu}
			Taking $\mu_1$ and $\mu_2$ in Theorem~\ref{th:ExpDec} satisfying
			\begin{equation*}
				\mu_1 < \min\left\lbrace \frac{1-2\beta-2\alpha^2}{2L\alpha^2}, \frac{1-2\alpha^2-3\beta}{L(2\alpha^2+\beta)} \right\rbrace
			\end{equation*}
			and
			\begin{equation*}
				\mu_2 < \min\left\lbrace \frac{1-2\beta-2(1+L\mu_1)\alpha^2}{2\beta}, \frac{1-2(1+L\mu_1)\alpha^2 - (1+L\mu_1)\beta -2\beta}{2\beta} \right\rbrace
			\end{equation*}
			the matrix $\Psi_{\mu_1,\mu_2}$, given by~\eqref{eq:Psi}, is negative definite.
		\end{remark}
		
		\subsection{General framework for stabilization}
		
		In the classical literature (See~\cite{Lions, Komornik}) a recognized estimate called~\emph{Observability} arises to address the exponential decay problem.  Note that once we are able to prove observability, the exponential stabilization holds. In fact, by employing the Proposition~\ref{pr:Diss} there exists $\mathcal{K}>0$ such that
		\begin{equation*}
			E'(t) +\mathcal{K}\left[ u_x^2(t,L) + u_x^2(t-h,L) + v_x(t,0)\right]\leq 0. 
		\end{equation*}
		Therefore, integrating in $[0,T]$ holds that
		\begin{equation*} 
			E(T) + \mathcal{K}\int_0^T u_x^2(t,L) + u_x^2(t-h,L) + v_x^2(t,0)\,dt \leq E(0).
		\end{equation*}
		Then, if we show that there exists a constant $C>0$ such that
		\begin{equation}\label{eq:Obs.0}
			E(0) \leq C \int_0^T u_x^2(t,L) + u_x^2(t-h,L) + v_x^2(t,0)\,dt
		\end{equation}
		the exponential stabilization yields. Indeed, we will obtain $E(T) - E(0)  \leq -C^{-1} E(0)$. Since the energy is dissipative, it follows that $E(T)\leq E(0)$, thus $E(T) - E(0)  \leq -C^{-1} E(T)$, which implies that
		$E(T) \leq \delta E(0),$ where $\delta=\frac{C}{1+C}<1.$
		Now, applying the same argument on the interval $[(m-1) T, m T]$ for $m=1,2, \ldots$, yields that
		$$
		E(m T) \leq \delta E((m-1) T) \leq \cdots \leq \delta^m E(0).
		$$
		Thus, we have
		$$
		E(m T) \leq e^{-\mu_0 m T} E(0) \quad \text { with } \quad \mu_0=\frac{1}{T} \ln \left(1+\frac{1}{C}\right)>0 .
		$$
		For an arbitrary $t>0$, there exists $m \in \mathbb{N}^*$ such that $(m-1) T<t \leq m T$, and by the non-increasing property of the energy, we conclude that
		$$
		E(t) \leq E((m-1) T) \leq e^{-\mu_0 (m-1) T} E(0) \leq \frac{1}{\delta} e^{-\mu_0 t} E(0),
		$$
		showing uniform exponential stability.
		
		Before to state and prove the observability inequality we need some \emph{a priori} estimates that allow us use the compactness-uniqueness argument to show~\eqref{eq:Obs.0}.
		\begin{proposition}
			Assume that $\alpha, \beta$ are real constant that satisfies~\eqref{eq:cond}. Let $L, T>0$ and $(u_0, v_0, z_0) \in H$. Then, the trace $v_x(\cdot,0)$ is well-defined, bounded and satisfies
			\begin{equation}\label{eq:trace1}
				\| v_x(\cdot,0)\|_{L^2(0,T)}^2 \leq C(\alpha,\beta) \left(\|(u_0,v_0)\|_{X_0}^2 + \|z_0(-h\cdot)\|_{L^2(0,1)}^2\right).
			\end{equation}
			Moreover, the standard energy and delay estimates holds, that is,
			\begin{equation}\label{eq:Kato.time}
				\begin{aligned}
					T\int_0^L u_0^2 + v_0^2\,dx& = \int_0^T\int_0^L u^2+v^2\,dx\,dt + \frac{1}{2} \int_0^T (T-t)u_x^2(t,L)\,dt \\
					&  - \int_0^T (T-t) \left[\alpha u_x(t,L) + \beta u_x(t-h,L)\right]^2\,dt + \int_0^T (T-t)v_x^2(t,0)\,dt
				\end{aligned} 
			\end{equation}
			and
			\begin{equation}\label{eq:Kato.delay}
				\|z_0(-h\cdot)\|_{L^2(0,1)}^2 \leq \| z(T,\cdot)\|_{L^2(0,1)}^2 + \frac{1}{h} \|z(\cdot,1)\|_{L^2(0,T)}^2.
			\end{equation}
		\end{proposition}
		\begin{proof}
			In order to show that the trace $v_x(\cdot,0)$ is well-defined and bounded, we employ the symmetric Morawetz multipliers. Indeed, multiplying~\eqref{eq:HS.lin}$_1$ by $xu$,~\eqref{eq:HS.lin}$_2$ by $(L-x)v$ and performing integration by parts in $(0,T)\times(0,L)$, we obtain
			\begin{multline*}
				0 = \frac{1}{2}\int_0^L \left[xu^2 +(L-x)v^2\right]_{t=0}^{t=T}\,dx - \frac{3}{4} \int_0^T\int_0^L u_x^2\,dx\,dt \\
				- \frac{3}{2} \int_0^T\int_0^L v_x^2\,dx\,dt 
				+ \frac{L}{4}\int_0^T u_x^2(t,L)\,dt + \frac{L}{2}\int_0^T v_x^2(t,0)\,dt.
			\end{multline*}
			Then, rearranging the terms and using Proposition~\ref{pr:Kato} holds that
			\begin{equation}
				\begin{aligned}
					\frac{L}{2}\int_0^T v_x^2(t,0)\,dt 
					& \leq  \frac{3}{4} \int_0^T\int_0^L u_x^2\,dx\,dt + \frac{3}{2} \int_0^T\int_0^L v_x^2\,dx\,dt  + \frac{L}{2}\int_0^L u_0^2 + v_0^2\,dx\\
					&\leq \left(\frac{3}{2}+\alpha^2+\beta^2\right)L\left(\|(u_0,v_0)\|_{X_0}^2 + \|z_0(-h\cdot)\|_{L^2(0,1)}^2\right)
				\end{aligned}
			\end{equation}
			Thus,~\eqref{eq:trace1} holds. To show~\eqref{eq:Kato.time}, we multiply~\eqref{eq:HS.lin}$_1$ and~\eqref{eq:HS.lin}$_2$  by $T-t$. Then,
			\begin{multline*}
				0 =  \frac{1}{2} \int_0^L \left[(T-t)(u^2+v^2)\right]_{t=0}^{t=T}\,dx + \frac{1}{2}\int_0^T\int_0^L u^2+v^2\,dx\,dt + \frac{1}{4}\int_0^T (T-t) u_x^2(t,L)\,dt \\
				-\frac{1}{2}\int_0^T (T-t) \left[\alpha u_x(t,L)+\beta u_x(t-h,L)\right]^2\,dt + \frac{1}{2} \int_0^T (T-t)v_x^2(t,0)\,dt.
			\end{multline*}
			Rearranging the terms the estimate yields.
			
			Finally, by multiplying~\eqref{eq:tr} by $z$ and integrating by parts we obtain
			\begin{equation*}
				\int_0^T u_x^2(t,0)\,dt + h \int_0^1 z_0^2(-h,\rho)\,d\rho = \int_0^T z^2(t,1)\,dt + h\int_0^1 z^2(T,\rho)\,d\rho.
			\end{equation*}
			as consequence of the estimate above,~\eqref{eq:Kato.delay} holds.
		\end{proof}
		
		Now we state and prove the observability inequality, first for the linear system~\eqref{eq:HS.lin}
		\begin{proposition}
			Let $\alpha,\beta$ real constant such that~\eqref{eq:cond} holds and $L>0$. Then, for all $T>h$, there exists $C= C(L,T) > 0 $ such that for every initial data $(u_0,v_0,z_0) \in H$,
			\begin{equation}\label{eq:Obs.1}
				\|(u_0,v_0,z_0)\|_{H}^2 \leq C \int_0^T u_x^2(t,L) + u_x^2(t-h,L) + v_x^2(t,0)\,dt
			\end{equation}
		\end{proposition}
		
		\begin{proof}
			Suppose that~\eqref{eq:Obs.1} does not holds. Then, there exists a sequence $\left((u_0,v_0,z_0)\right)_n \subset H$ such that
			\begin{equation}\label{eq:Obs.2}
				\begin{aligned}
					1&=\int_0^L(u_0^n)^2+(v_0^n)^2\,dx + \beta h \int_0^1 (z_0^n)^2(t,\rho)\,d\rho \\
					& > n \int_0^T (u_x^n)^2(t,L) + (z^n)^2(t,1) + (v_x^n)^2(t,0)\,dt 
				\end{aligned}
			\end{equation}
			where $(u^n,v^n,z^n) = S(\cdot)(u_0^n, v_0^n, z_0^n)$ and $S(\cdot)$ denotes the associated semigroup generated by the operator $A$. Follows from Proposition~\ref{pr:Kato} and~\eqref{eq:Obs.2} that $(u^n,v^n)_n$ is bounded in $L^2(0,T; [H^1(0,L)]^2)$. Moreover, using~\eqref{eq:HS.lin}, $(u_t^n,v_t^n)_n$ is bounded in $L^2(0,T; [H^{-2}(0,L)]^2)$. As $[H^1(0,L)]^2 \subset X_0 \subset [H^{-2}(0,L)]^2$, follows that by applying the Aubin-Lions Theorem that exists $(u^n,v^n)_n$ sequence relatively compact in $L^2(0,T ; X_0)$, that is, there exists a subsequence, still denoted $(u^n,v^n)_n$, that converges strongly to $(u,v)$ in $L^2(0,T ; X_0)$.
			Using~\eqref{eq:Kato.time} joint with~\eqref{eq:Obs.2} yields that $(u_0^n,v_0^n)_n$ is a Cauchy sequence in $X_0$. 
			
			On the other hand~\eqref{eq:Kato.delay} implies that $(z_0^n(-h\cdot))_n \subset L^2(0,1)$ is a Cauchy sequence in $L^2(0,1)$ whenever $T>h$. In fact, as $z^n(T,\rho) = u_x^n(T-\rho h)$, for $T>h$ follows that 
		$$\int_0^1 (z^n)^2(T,\rho)\,d\rho \leq \frac{1}{h}\int_0^T (u_x^n)^2(t,L)\,dt.$$
Employing~\eqref{eq:Kato.delay}, yields that
	$$\|z_0^n(-h\cdot)\|_{L^2(0,1)}^2 \leq \frac{1}{h}\|u_x^n(\cdot,L)\|_{L^2(0,T)}^2 + \frac{1}{h}\|z^n(\cdot,1)\|_{L^2(0,T)}^2. $$
	That is, $(z_0^n(-h\cdot))_n \subset L^2(0,1)$ is a Cauchy sequence in $L^2(0,1)$.
			
			Consequently, $(u_0,v_0,z_0) = \lim_{n\to\infty} (u_0^n,v_0^n,z_0^n)$ in $H$ and $(u,v,z) = S(\cdot)(u_0,v_0,z_0)$ and by Proposition~\ref{pr:Kato}, yields that
			\begin{equation*}
				\begin{cases}
					u_x^n(\cdot, L) \to u_x(\cdot,L) ,\\
					z^n(\cdot,1) \to z(\cdot,1),\\
					v_x^n(\cdot,0) \to v_x(\cdot,0)
				\end{cases}
				\quad\text{in }L^2(0,T).
			\end{equation*}
			
			We infer from~\eqref{eq:Obs.2} that $u_x(\cdot,L) = u_x(\cdot-h,L) = v_x(\cdot,0) = 0$ and thus $v_x(\cdot,L)=0$.
			Summarizing, the $(u,v)$ is solution of the linearized Hirota-Satsuma system
			\begin{equation}\label{eq:Obs.3} 
				\begin{cases}
					u_t - \frac{1}{2} u_{x x x} = 0 & x \in (0,L),\  t>0 \\
					v_t+v_{x x x} = 0 & x \in (0,L),\  t>0 \\
					u(t,0) = u(t,L) = u_x(t,0) = u_x(t,L)=0, & t>0, \\
					v(t,0) = v(t,L) = v_x(t,0) = 0, & t > 0, \\
					v_x(t,L) = \alpha u_x(t,L) + \beta u_x(t-h,L) = 0, & t>0,\\
					u(0,x) = u_0(x)\in L^2(0,L) , v(0,x) = v_0(x) \in L^2(0,L) 
				\end{cases}
			\end{equation} 
			with $\|(u_0,v_0)\|_{X_0} = 1$. Due de structure of the linearized Hirota-Satsuma system, we can recast~\eqref{eq:Obs.3} as
			\begin{equation}\label{eq:Obs.4} 
				\begin{cases}
					u_t - \frac{1}{2} u_{x x x} = 0, \\
					u(t,0) = u(t,L) = u_x(t,0) = 0,  \\
					u_x(t,L)= 0 , \\
					u(0,x) = u_0(x) \in L^2(0,L) ,
				\end{cases}
				\quad\text{and}\quad
				\begin{cases}
					v_t+v_{x x x} = 0, \\
					v(t,0) = v(t,L) =  v_x(t,L) = 0, \\
					v_x(t,0) = 0, \\
					v(0,x) = v_0(x) \in L^2(0,L) .
				\end{cases}
			\end{equation} 
			Then, by multiplying each equation by $u$ and $v$ respectively, performing some integration by parts and using the boundary conditions, we conclude that the systems in~\eqref{eq:Obs.4} only admits the trivial solution. Hence, we obtain a contradiction and~\eqref{eq:Obs.1} holds.
		\end{proof}
		
		Now, we can address the stabilization problem for the Hirota-Satsuma system~\eqref{eq:HS}.
		\begin{proof}[Proof of Theorem~\ref{th:ES.2}]
			Let $(u_0,v_0, z_0) \in H$ such that $\|(u_0,v_0, z_0)\|_H \leq R$, where $R$ will be chosen later. Observe that the solution of~\eqref{eq:HS} can be written as $(u,v)=(u_1,v_1) + (u_2,v_2)$ where the pair $(u_1,v_1)$ is the solution of
			\begin{equation}\label{eq:HS.1}
				\begin{cases}
					u_{1,t}-\frac{1}{2} u_{1,x x x}= 0 & x \in (0,L),\  t>0, \\
					v_{1,t}+v_{1,x x x} = 0 & x \in (0,L),\  t>0, \\
					u_1(t,0) = u_1(t,L) = v_1(t,0) = v_1(t,L) = u_{1,x}(t,0) = 0, & t>0, \\
					v_{1,x}(t,L) = \alpha u_{1,x}(t,L) + \beta u_{1,x}(t-h,L), & t>0,\\
					u_1(0,x) = u_0(x), v_1(0,x) = v_0(x), &  x \in (0,L) \\
					u_{1,x}(t, L) = z_0(t),  t \in (0,1).		
				\end{cases}
			\end{equation}
			and
			\begin{equation}\label{eq:HS.2}
				\begin{cases}
					u_{2,t}-\frac{1}{2} u_{2,x x x}= 3uu_x +3vv_x & x \in (0,L),\  t>0, \\
					v_{2,t}+v_{2,x x x} = -3uv_x & x \in (0,L),\  t>0, \\
					u_2(t,0) = u_2(t,L) = v_2(t,0) = v_2(t,L) = u_{2,x}(t,0) = 0, & t>0, \\
					v_{2,x}(t,L) = \alpha u_{2,x}(t,L) + \beta u_{2,x}(t-h,L), & t>0,\\
					u_2(0,x) = 0, v_2(0,x) = 0, &  x \in (0,L) \\
					u_{2,x}(t, L) = 0,  t \in (0,1).		
				\end{cases}
			\end{equation}
			
			Observe that, $u_1$ is the solution of the linear problem~\eqref{eq:HS.lin} with initial data $(u_0,v_0,z_0) \in H$ and $u_2$ is the solution of the problem~\eqref{eq:HS.Sou} with null initial data and sources $(f_1,f_2) = (3(uu_x+vv_x), -3uv_x) \in L^1(0,T ; X_0)$. Then, by the energy dissipation, Theorem~\ref{th:KatoSou} and Proposition~\ref{pr:nl.Sou} follows that
			\begin{equation}\label{eq:ES.nl}
				\begin{aligned}
					\| (u(T), v(T), z(T)) \|_{H} 
					& \leq \|(u_1(T),v_1(T), z_1(T))\|_{H} + \|(u_2(T),v_2(T), z_2(T))\|_{H} \\
					& \leq \delta  \|(u_0,v_0, z_0(-h\cdot))\|_{H} +  C \| (uu_x+vv_x, uv_x)\|_{L^1(0,T ; X_0)} \\
					& \leq \delta  \|(u_0,v_0, z_0(-h\cdot))\|_{H} +  C \| (u,v)\|_{L^2(0,T ; [H^1(0,L)]^2)}^2
				\end{aligned}
			\end{equation}
			with $0<\delta<1$. To estimate the last term of the inequality above, we proceed as much as in~\eqref{eq:Kato2}. By multiplying~\eqref{eq:HS}$_1$ by $(L-x)u$ and~\eqref{eq:HS}$_2$ by $xv$, integrating by parts, rearranging the terms and using the boundary conditions follows that
			\begin{multline}\label{eq:nl.0}
				\frac{3}{4} \int_0^T\int_0^L u_x^2\,dx\,dt + \frac{3}{2} \int_0^T\int_0^L v_x^2\,dx\,dt 
				\leq C(L,\alpha,\beta)  \|(u_0,v_0, z_0(-h\cdot))\|_{H}^2 \\+  \int_0^T\int_0^L u^3 \,dx\,dt +   3\int_0^T\int_0^L (L-2x) uvv_x\,dx\,dt.
			\end{multline}
			Then, by using the Gagliardo-Nirenberg and Young inequalities we can estimate the nonlinear terms as
			\begin{equation}\label{eq:nl.1}
				\begin{aligned}
					\int_0^T\int_0^L u^3\,dx\,dt +& 3\int_0^T\int_0^L (L-2x)uvv_x\,dx\,dt   \leq  \max\left\lbrace \frac{9CL^2}{4}, 1\right\rbrace \|(u,v)\|_{L^2(0,T ;[H^{1}(0,L)]^2}^2 \\
					& +\frac{CL T}{2} \|(u_0,v_0, z_0(-h\cdot))\|_{H}^4 + \frac{27CL^2T}{8} \|(u_0,v_0, z_0(-h\cdot))\|_{H}^2  
				\end{aligned}
			\end{equation}

			Therefore, gathering~\eqref{eq:nl.0} and~\eqref{eq:nl.1} follows that there exists a constant $\mathcal{C}>0 $ such that
			\begin{equation}
				\int_0^T\int_0^L u_x^2 + v_x^2\,dx\,dt \leq \mathcal{C} \left(  2 \|(u_0,v_0, z_0(-h\cdot))\|_{H}^2 +  \|(u_0,v_0, z_0(-h\cdot))\|_{H}^4 \right).
			\end{equation}
			Thus, from~\eqref{eq:ES.nl} there exists $C>0$ such that
			\begin{multline*}
				\| (u(T), v(T), z(T)) \|_{H} \leq \\ \|(u_0,v_0, z_0(-h\cdot))\|_{H} \left(\delta + (2+C) \|(u_0,v_0, z_0(-h\cdot))\|_{H} + C\|(u_0,v_0, z_0(-h\cdot))\|_{H}^3\right)
			\end{multline*}
			which implies
			\begin{equation}
				\begin{aligned}
					\| (u(T), v(T), z(T)) \|_{H} & \leq \|(u_0,v_0, z_0(-h\cdot))\|_{H} \left(\delta + (2+C)R + C R^3\right) \\
					& \leq \|(u_0,v_0, z_0(-h\cdot))\|_{H} \left(\delta + \varepsilon\right)
				\end{aligned}
			\end{equation}
			whenever $R$ is taken small enough such that $(2+C)R + C R^3 < \varepsilon$ with $\varepsilon>0$ small enough to obtain $\delta + \varepsilon < 1$. Consequently,
			\begin{equation}
				\| (u(T), v(T), z(T)) \|_{H} \leq (\delta+\varepsilon) \|(u_0,v_0, z_0(-h\cdot))\|_{H}
			\end{equation}
			with $\delta + \varepsilon < 1$. Thus, the result follows from an inductive argument analogous to the linear case.
		\end{proof}
		
		\section*{Acknowledge}
		Muñoz acknowledges support from FACEPE grant IBPG-0909-1.01/20 and  this research is part of their Ph.D. thesis at the Department of Mathematics of the Universidade Federal de Pernambuco. Gonzalez Martinez was supported by CAPES/COFECUB grant 88887.879175/2023-00 and CNPq grant 421573/2023-6.

	\end{document}